\documentclass[12pt, a4paper]{article}
\usepackage{graphicx}
\usepackage{amssymb}
\usepackage{stmaryrd}
\usepackage{enumitem}

\usepackage{geometry}
\usepackage{tikz}
\usepackage{float}
\usepackage{url}
\usepackage{mathtools}
\usepackage{amsmath,amsfonts,amsthm}
\usepackage{mathrsfs} 

\newcommand{\claim}[1]{\paragraph*{Claim #1:}}

\newtheorem{theorem}{Theorem}[section]
\newtheorem{lemma}[theorem]{Lemma}
\newtheorem{proposition}{Proposition}[section]
\newtheorem{corollary}[theorem]{Corollary}
\newtheorem{definition}{Definition}

\newtheorem{remark}{Remark}

\usepackage{pgfplots}
\pgfplotsset{compat=1.18}
\newcommand{\TT}{\mathcal{T}}
\newcommand{\irr}{\mathrm{irr}}
\newcommand{\SO}{\mathrm{SO}}

\usepackage{subfig}

\DeclarePairedDelimiter\floor{\lfloor}{\rfloor}

\usepackage{ytableau}
\usepackage[colorlinks,
linkcolor=blue,
anchorcolor=blue,
citecolor=blue
]{hyperref}
\begin{document}
\begin{center}
{\large \bf  Theta-Relations Among Degree-Based Tree Indices}
\end{center}
\begin{center}
 Jasem Hamoud$^1$ \hspace{0.2 cm}  Duaa Abdullah$^{2}$\\[6pt]
 $^{1,2}$ Physics and Technology School of Applied Mathematics and Informatics \\
Moscow Institute of Physics and Technology, 141701, Moscow region, Russia\\[6pt]
Email: 	 $^{1}${\tt hamoud.math@gmail.com},
 $^{2}${\tt abdulla.d@phystech.edu}
\end{center}
\noindent
\begin{abstract}
In this paper, degree-based topological indices play a key role in the structural analysis of graphs in this paper and have significant uses in chemical graph theory. We investigate the connections between three such tree indices: the Albertson, Sombor, and Sigma indices.
We show that the quadratic degree deviation, measured by the Sigma index, tightly controls the Sombor index of a tree by establishing sharp two-sided bounds. We demonstrate that the Sombor and Sigma indices are asymptotically equivalent up to constant factors as a direct result. A pure $\Theta$-relationship between the Sombor index and the Albertson index is derived by taking into account extremal trees with a fixed degree sequence.  This finding demonstrates that, in extremal configurations, quadratic degree interactions and absolute degree disparities scale appropriately. Overall, our data suggest that the Sombor index functions as an intermediate descriptor, capturing both global degree dispersion and local edge irregularity. 
From a structural standpoint, these findings clarify the relationship between vertex-based and edge-based irregularity measurements in trees.
\end{abstract}

\noindent\rule{16cm}{1.0pt}

\noindent\textbf{AMS Classification 2010:} 05C05, 05C12, 05C20, 05C25, 05C35, 05C76, 68R10.


\noindent\rule{16cm}{1.0pt}

\section{Introduction}
Throughout this paper, let $G=(V,E)$ be a simple connected graph with vertex set $V(G) = \{v_1, v_2, \dots, v_n\}$, where $n = |V(G)|$, and edge set $E(G) = \{e_1, e_2, \dots, e_m\}$, where $m = |E(G)|$. Let $\mathscr{D} = (d_1, d_2, \dots, d_n)$ be the degree sequence of $G$, where $d_w = d_G(v_w)$ is the degree of vertex $v_w \in V(G)$, arranged in non-increasing order. S.~L.~Hakimi~\cite{Hakimi1962} introduced the concept of realizability for a sequence of nonnegative integers as the degree sequence of a simple graph. 
Caro and Pepper~\cite{Caro2014Pepper} studied degree sequence indices and related strategies.
Several works~\cite{Schmidt1976Druffel,Baskin1987Gordeeva} have addressed the inverse problem for topological indices. One such index, the \emph{Albertson index}, was introduced in 1997 by M.~O.~Albertson~\cite{albertson1997irregularity} to quantify graph irregularity. Gutman~\cite{Gutman2018I} surveyed topological indices and irregularity measures, presenting several equivalent formulas for the Albertson index (also referred to as a measure of \emph{irregularity}) in~\cite{abdo2014total,Abdo2018Dimitrov,Albalahi2023Alanazi,Alex2024Indulal,Akg2019Naca,Arif2023Hayat}.

The Albertson index is generally defined as
$$
\operatorname{irr}(G) = \sum_{uv \in E(G)} |d_G(u) - d_G(v)|.
$$

This is not the only possible irregularity measure. Indeed, Criado et al.~\cite{Criado2014Flores} proposed an alternative irregularity index given by
$$
\operatorname{irr}(G) = \frac{1}{n} \sum_{i=1}^{n} \left| d_i - \frac{2m}{n} \right|.
$$
The \textit{Sigma index} (also called the $\sigma$-irregularity index), denoted $\sigma(G)$, is a degree-based topological index that measures the irregularity of a graph $G$. It is defined as
$$
\sigma(G)=\sum_{uv \in E(G)} (d_G(u)-d_G(v))^2,
$$
where the sum is over all edges and $d_G(u)$ is the degree of vertex $u$~\cite{Abdo2015Dimitrov,gutman2018inverse}.
An equivalent expression~\cite{Criado2014Flores}, related to the variance of the degree sequence, is
$$
\sigma(G) = \frac{1}{n} \sum_{i=1}^n \left( d_i - \frac{2m}{n} \right)^2.
$$
Relations to other topological indices appear in~\cite{Jahanbai2021Sheikholeslami}.
One of the most influential degree-based topological indices is the Sombor index, introduced by Gutman~\cite{Gutman2021Geo} approximately five years ago. For a graph $G$, the Sombor index is defined as
\[
\SO(G)=\sum_{uv\in E(G)}\sqrt{d_G(u)^2+d_G(v)^2}.
\]

Chemical graph theory, a subfield of mathematical chemistry, employs graph-theoretical methods to model and analyze molecular structures~\cite{Dac2026KCV}. Consider molecules are represented as graphs in which vertices correspond to atoms and edges represent chemical bonds.  Molecular descriptors --- most notably topological indices when defined on molecular graphs --- serve as essential tools in virtual screening and in the prediction of physicochemical properties of chemical compounds~\cite{Basak2024,Desmecht2024}.

Establishing rigorous structural relationships between the Sigma index, Sombor index, and Albertson index—three basic degree-based topological indices on trees—is the main goal of this work. With special attention to extremal configurations and fixed degree sequences, we elucidate the mathematical relationships among vertex-based degree dispersion (Sigma), edge-based degree interaction (Sombor), and degree-difference irregularity (Albertson) by creating a unified analytical framework.

There are various reasons why this study is important. First, while the Sigma, Sombor, and Albertson indices have all been studied separately, their interrelationships have not been thoroughly examined, particularly with regard to asymptotic growth rates and extreme behavior. The current findings show how these indices scale together and pinpoint the structural circumstances in which they become asymptotically equivalent up to constant factors. They also provide sharp two-sided bounds and pure $\Theta$-relations.
Second, the study shows that the Sombor index bridges the gap between vertex irregularity and edge irregularity by functioning as an intermediary descriptor. The Sombor index incorporates both the magnitude and the contrast of endpoint degrees across edges, whereas the Albertson index measures local degree contrasts and the Sigma index is only dependent on the degree sequence. 

Furthermore, the established relationships enhance the interpretability of Sombor-type descriptors used in QSPR/QSAR modeling in the context of chemical graph theory, where trees represent acyclic molecular structures. The findings make it easier to estimate, compare, and interpret degree-based molecular descriptors structurally by offering a theoretical connection between these indices. All things considered, this work contributes to the theoretical knowledge of degree-based irregularity measures and establishes a logical framework for further applied and extreme research.

 \section{Preliminaries}\label{sec2}
In this section, understanding how degree distributions constrain the global structure of trees is a central problem in extremal graph theory and its applications. Many important graph invariants---such as degree variance, irregularity indices, and Zagreb-type indices---are governed by the interplay between local vertex degrees and global structural parameters. In this context, we bound the degree--square deviation.

Several classical indices, such as the Zagreb indices, are also used to quantify structural features of graphs, particularly in chemical graph theory.
\begin{definition}[Zagreb Indices~\cite{gutman1972total, gutman1975acyclic}]
Let $ G = (V, E) $ be a graph. The first and second Zagreb indices are given by
\[
M_1(G) = \sum_{v \in V(G)} d_G(v)^2, \quad M_2(G) = \sum_{uv \in E(G)} d_G(u)d_G(v).
\]
\end{definition}
The following proposition provides a lower bound on the Albertson irregularity index of a graph in terms of its extremal degrees, shedding light on the impact of degree disparity.

\begin{proposition}[\cite{dorjsembe2022irregularity}]
Let $G$ be a graph with minimum degree $\delta$, maximum degree $\Delta$, and $n$ vertices. Then
$$
\operatorname{irr}(G) > \frac{\delta(\Delta - \delta)^2 \cdot n}{\Delta + 1}.
$$
\end{proposition}

The following lemma summarizes the extremal behavior of two well-known irregularity measures---the Albertson index and the Sigma index---within the class of trees. These bounds provide a reference for understanding the behavior of vertex-based and edge-based irregularity measures in extremal tree configurations and serve as a benchmark for comparison with more structured families of trees (such as caterpillars) studied later in this work.

\begin{lemma}[\cite{fath2013extremely,gutman2018inverse}]
Let $T$ be a tree on $n \geq 3$ vertices. Then
\begin{enumerate}
\item the Albertson irregularity index satisfies
  $$
  \operatorname{irr}(T) \le (n-2)(n-1),
  $$
  with equality if and only if $T$ is the star $S_n$;

  \item the Sigma index satisfies
  $$
  0 \le \sigma(T) \le (n-1)(n-2),
  $$
  where $\sigma(T) = 0$ if and only if $T$ is regular (i.e., $T \cong P_2$), and $\sigma(T) = (n-1)(n-2)$ if and only if $T$ is the star $S_n$.
\end{enumerate}
\end{lemma}

Several important results concerning the Sombor index that are essential for understanding the findings presented later.

\begin{theorem}[\cite{Gutman2021Geo}]
Let $K_n$ denote the complete graph on $n$ vertices and let $\overline{K_n}$ denote its complement (the empty graph on $n$ vertices). Then for every graph $G$ of order $n$,
$$
\SO(\overline{K_n}) \leq \SO(G) \leq \SO(K_n),
$$
with equality if and only if $G \cong \overline{K_n}$ or $G \cong K_n$.  
\end{theorem}

\begin{theorem}[\cite{Gutman2021Geo,Liu2023GutmanYou}]
Let $P_n$ denote the path on $n$ vertices. Then for every connected graph $G$ of order $n$,
$$
\SO(P_n) \leq \SO(G) \leq \SO(K_n),
$$
with equality if and only if $G \cong P_n$ or $G \cong K_n$. For $n \geq 3$,
$\SO(P_n)=2\sqrt{5} + 2(n-3)\sqrt{2}$.
\end{theorem}

\subsection{Problem Statement}
Assume $\TT_{n,\Delta}$ be the family of tree of order $n$ with the maximum degree $\Delta$. Let $\mathscr{D}=(d_1,d_2,\dots,d_n)$ be a degree sequence where $d_n\geqslant d_{n-1}\geqslant \dots,\geqslant d_2\geqslant d_1$. Let $\eta$ be an integer defined as 
\begin{equation}
\eta=\floor*{\frac{kn}{2}}, \quad k=\frac{1}{n-\Delta}\sum_{i=1}^{n}d_i^2.
\end{equation}
Let $\TT\in \TT_{n,\Delta}$ be a tree given by Figure~\ref{familytren01r} as 
\begin{figure}[H]
    \centering
 \begin{tikzpicture}[scale=.8]
\draw  (3,7)-- (5,7);
\draw [dotted] (5,7)-- (7,7);
\draw  (7,7)-- (9,7);
\draw  (9,7)-- (10,8);
\draw  (9,7)-- (10,6);
\draw  (3,7)-- (2,8);
\draw  (3,7)-- (2,6);
\draw (3,7) node[anchor=north west] {$v_1$};
\draw (8.3,7) node[anchor=north west] {$v_n$};
\draw (9.9,7.54) node[anchor=north west] {$\vdots$};
\draw (1.9,7.58) node[anchor=north west] {$\vdots$};
\begin{scriptsize}
\draw [fill=black] (3,7) circle (1.5pt);
\draw [fill=black] (5,7) circle (1.5pt);
\draw [fill=black] (7,7) circle (1.5pt);
\draw [fill=black] (9,7) circle (1.5pt);
\draw [fill=black] (10,8) circle (1.5pt);
\draw [fill=black] (10,6) circle (1.5pt);
\draw [fill=black] (2,8) circle (1.5pt);
\draw [fill=black] (2,6) circle (1.5pt);
\end{scriptsize}
\end{tikzpicture}
\caption{An example of tree $\TT\in \TT_{n,\Delta}$.}
\label{familytren01r}
\end{figure}
Hence, for $\TT_1$ we have $d(v_1)=d(v_n)=\lambda_1$, $d(v_2)=d(v_3)=\dots=d(v_{n-1})=\mu_1$, for $\TT_2$ we have $d(v_1)=d(v_n)=\lambda_2$, $d(v_2)=d(v_3)=\dots=d(v_{n-1})=\mu_2$, with 
\begin{equation}
\mu=\floor*{\frac{\eta-2}{\eta-n}}, \quad \floor*{\frac{\eta-k}{\eta-n}}=1.
\end{equation}
where the sequences $\lambda=(\lambda_1,\lambda_2,\dots,\lambda_n)$, $\mu=(\mu_1,\mu_2,\dots,\mu_n)$ and $1\leqslant \mu \leqslant 2$. The question is: what is the relationship between Sombor index and Sigma index especially, 
\[
|SO(\TT)|\asymp \sigma(\TT)+\frac{4m^2}{n}
\]
Not directly proportional to $\sigma(\TT)$.

\paragraph*{Highlights}
\begin{itemize}
  \item Sharp two-sided bounds are established that relate the Sombor index and the Sigma index for trees.
  \item A precise $\Theta$-relationship between the Sombor index and the Albertson index is derived for extremal trees.
  \item The Sombor index is shown to interpolate between vertex-based and edge-based irregularity measures.
  \item Structural characterizations of the extremal trees are provided for any fixed degree sequence.
  \item Implications for molecular graph descriptors and degree-based topological indices are discussed.
\end{itemize}

\section{Main Result}
In this section, we start with Proposition~\ref{fibpron1task} which it is crucial as it provides tight bounds on the degree irregularity of trees using solely global parameters, identifies extremal tree structures, and offers a stable inductive framework for analyzing degree-based topological invariants.

\begin{proposition}~\label{fibpron1task}
Let $\TT\in \TT_{n,\Delta}$ be a tree of order $n$ with integers $\mu$ and $\eta$, $\Delta$ is the maximum degree of $\TT$. Then, the Albertson index and Sigma index satisfy
\begin{equation}~\label{eqq1fibpron1task}
\floor*{\frac{3(\eta-n)^2}{k\Delta}}  \leqslant \sigma(\TT) \leqslant \eta(2n\mu^2+\eta\mu\,(\mu-1)\,\irr(\TT)).
\end{equation}
\end{proposition}
\begin{proof}
Recall $\TT$ be a tree of order $n$ with integers $\mu$, $\eta$, $\Delta$. We prove our argument~\eqref{eqq1fibpron1task} in two steps: first, we deduce the upper bound and then establish the lower bound. For $n=3$, the unique tree is the path $\mathcal{P}_3$ with degree sequence $(1,2,1)$. Thus, $\sigma(\mathcal{P}_3)\leqslant \eta(2n\mu^2).$ Thus, the inequality holds for $n=3$. Consider $\Delta\geqslant 3$ and according to the term of $\mu$ it satisfied $1\leqslant \mu \leqslant 2$. Assume that for every tree $\TT\in\TT_{n,\Delta}$, the relationship~\eqref{eqq1fibpron1task} holds.  Thus, $\mu\,(\eta-n)+2\leqslant \eta \leqslant (\mu+1)\,(\eta-2)+2$. Since $\eta>n$ and $\sigma(\TT)>2n^2-2$, we find that 
\begin{equation}~\label{eqq2fibpron1task}
\sigma(\TT)\leqslant  \mu\,(\eta-n)-2n^2+2.
\end{equation}
Let $\TT' \in \TT_{n+1,\Delta}$. Since $\TT'$ is a tree, it has a pendant vertex $u$ adjacent to a vertex $v$. Let $\TT=\TT'-u$. Then $d_{\TT}(v)=d_{\TT'}(v)-1$ and $d_{\TT'}(u)=1.$ 
Therefore, for the degrees $d(x)\geqslant d(y) \geqslant d(z) \geqslant 3$ consist of tree $\TT'$ drived from $\TT$ be omitted the edges $wy$ and $xw$ where $\TT'=\TT-yw+xw$. Thus, consider $\Delta>d(x)$. Then, 
\begin{align*}
\irr(\TT)-\irr(\TT')&=\sum_{uv\in E(\TT)}|d(u)-d(v)|-\sum_{xy\in E(\TT')}|d(x)-d(y)|\\
&=|d(y)-d(w)|+|d(x)-d(w)|-|d(y)-d(w)-1|-|d(x)-d(w)+1|\\
&=2(d_\TT(x)-d_\TT(y))\\
&\geqslant \mu \left(|d(y)-d(w)|+|d(x)-d(w)|\right).
\end{align*}
Therefore, $\irr(\TT)-\irr(\TT')\geqslant 2$, thus, $\irr(\TT)>\irr(\TT')$. Since $\sqrt{\sigma(\TT)}\leqslant \irr(\TT)\leqslant \sqrt{m\,\sigma(\TT)}$ and $\sigma(\TT)>2\,\irr(\TT)$ implies that trees $\TT$ and $\TT'$ satisfying the Sigma index condition also satisfy $\sigma(\TT)-\sigma(\TT')\geqslant 4$, thus, $\sigma(\TT)>\sigma(\TT')$. Since $d(v) \leqslant \Delta$, we have
$k'\leqslant k+\frac{2\Delta+2}{n-\Delta}.$
Thus,
$$
\sigma(\TT')=\sigma(\TT)+(d(v)+1-d^{\prime}(v))^2-(d(v)-d^{\prime}(v)^2+(1-d^{\prime}(v)^2).
$$
Thus,
\begin{equation}~\label{eqq03fibpron1task}
\sigma(\TT') \leqslant \sigma(\TT)+2\eta\mu^2+\eta\mu(\mu-1)\irr(\TT'),
\end{equation}
from~\eqref{eqq2fibpron1task} and \eqref{eqq03fibpron1task}, we find that $\sigma(\TT)>(\mu\,(\eta-n))/(2n^2-2)$. Then,
\begin{equation}~\label{eqq3fibpron1task}
\sigma(\TT)\geqslant \frac{\mu\,(\eta-n)}{2n^2-2}+\frac{(\eta-n)^2}{k\Delta}.
\end{equation}
Hence, from~\eqref{eqq2fibpron1task}--\eqref{eqq3fibpron1task} the left hand of the relationship~\eqref{eqq1fibpron1task} holds. 
Since $\irr(\TT)\leqslant \sqrt{m\,\sigma(\TT)}$ and $\mu(\mu-1)<m$ implies that $\sigma(\TT)>m\,\irr(\TT)$. Thus, $\sigma(\TT)\geqslant \mu\,(\mu-1)\,\irr(\TT)+\Delta(\Delta-1)^2$. Then 
\begin{equation}~\label{eqq4fibpron1task}
\sigma(\TT)\geqslant \frac{\mu\,(\eta-n)}{2n^2-2}+\eta\mu\,(\mu-1)\,\irr(\TT).
\end{equation}
Also, for $n$ it implies that $n>m>\mu^2$ and $2n^2-2>\eta\,\mu$. As a result $\sigma(\TT)>2n^2-2+\eta\,\mu$, by considering $\irr(\TT)-\irr(\TT')\geqslant \mu \left(|d(y)-d(w)|+|d(x)-d(w)|\right),$ according to~\eqref{eqq3fibpron1task} and \eqref{eqq4fibpron1task} we find that 
\begin{equation}~\label{eqq5fibpron1task}
\sigma(\TT) \geqslant 2n\mu^2+\eta\mu\,(\mu-1)\,\irr(\TT),
\end{equation}
yields
$$
\sigma(\TT') \leqslant \eta\bigl(2(n+1)\mu^2+\eta\mu(\mu-1)\irr(\TT')\bigr).
$$
The lower bound follows from the monotonicity of $\eta$ and $k$. Hence, the inequality holds for $n+1$. Thus, from~\eqref{eqq5fibpron1task} the right hand of relationship~\eqref{eqq1fibpron1task} holds.
\end{proof}
For example, let $\mathscr{D}=(2^{11}, 10)$ be a degree sequence given by Figure~\ref{figxxxLowern1}. Then, $\sigma(\TT)=794$, $\irr(\TT)=90$, lower bound is 573 and the upper bound is 8208.
\begin{figure}[H]
    \centering
    \includegraphics[width=0.5\linewidth]{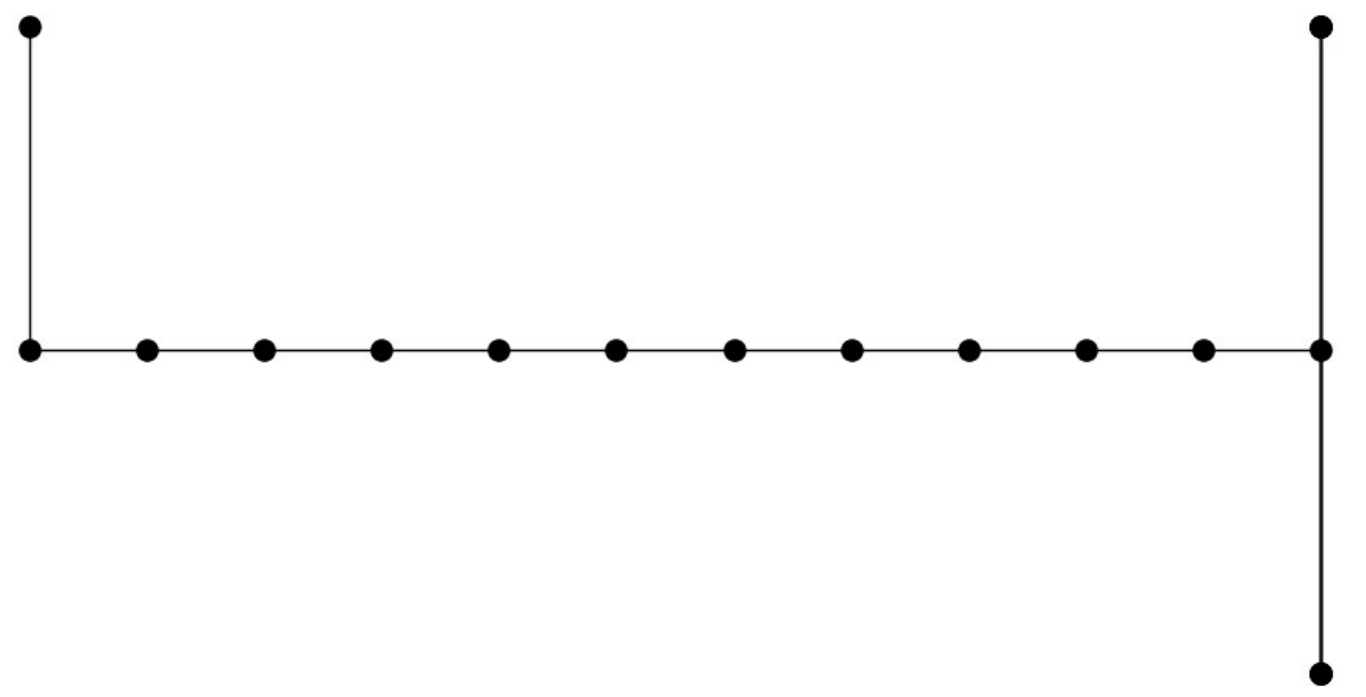}
    \caption{Lower and upper bound in Proposition~\ref{fibpron1task}.}
    \label{figxxxLowern1}
\end{figure}
\begin{proposition}~\label{fibpron2task}
Among all trees $\TT \in \TT_{n,\Delta}$, $\irr(\TT)<\eta<\sigma(\TT)$.
\end{proposition}
\begin{proof}
Assume the tree $\TT$ has degree sequence $(d_1 \leqslant \dots \leqslant d_n)$. Hence, $\TT'$ has degree sequence $(d_1, \dots, d_{j-1}, d_j-1, \dots, d_{k-1}, d_k+1, d_{k+1}, \dots, d_n),$ where $2\leqslant j\leqslant k \leqslant n$. Let $\TT' \in \TT_{n+1,\Delta}$. According to Proposition~\ref{fibpron1task}, the Albertson index of $\TT'$ is 
\begin{equation}~\label{eqq1fibpron2task}
\irr(\TT')=\lambda_1\,(\mu+1)\mu\,n(\mu(\eta-n)-n+2) \\
\end{equation}
where $\lambda_1=(\eta-\mu(\eta-n)-2)(\mu(\eta-n)-n+2)n(\eta-\mu(\eta-n)-2)$. 
Thus, the sigma index of $\TT'$ is
\begin{equation}~\label{eqq2fibpron2task}
   \sigma(\TT')=\lambda_1\,(\mu+1)^2+n(\mu(\eta-n)-n+2)\mu^2. 
\end{equation}
Thus, from~\eqref{eqq1fibpron2task} and \eqref{eqq2fibpron2task} we obtain on $\irr(\TT)<\eta<\sigma(\TT)$.
\end{proof}
Actually, the relationship in Proposition~\ref{fibpron2task} is critical for all trees $\TT \in \TT_{n,\Delta}$, for example, let $\mathscr{D}=(2^{18},3,10)$ be a degree sequence. Then,  $\sigma(\TT)=5178$, $\irr(\TT)=310$, $\eta=1755$, lower bound is 2858 and the upper bound is 70200.

\begin{lemma}~\label{lemSomborIn1}
Let $\TT\in \TT_{n,\Delta}$ be a tree of order $n$ with integers $\mu$ and $\eta$, $\Delta$ is the maximum degree of $\TT$. Consider $\TT'$ be the extremal tree with degree sequence
\[
\mathscr{D}=(\underbrace{\mu+2,\dots,\mu+2}_{x},
 \underbrace{\mu+1,\dots,\mu+1}_{y},
 \underbrace{1,\dots,1}_{n}),
\]
where $x=\eta-\mu(\eta-n)-2$ and $y=\mu(\eta-n)-n+2$. Then the Sombor index of $\TT'$ is given by
\begin{equation}~\label{eqq1lemSomborIn1}
\mathrm{SO}(\TT')=xy\sqrt{2\mu^2+6\mu+5}+xn\sqrt{\mu^2+4\mu+5}+yn\sqrt{\mu^2+2\mu+2}.
\end{equation}
Moreover, for any tree $\TT\in \TT_{n,\Delta}$ satisfying $\mathrm{SO}(\TT)\leqslant \mathrm{SO}(\TT')$ if and only if $\TT\cong \TT'$.  
\end{lemma}
\begin{proof}
Assume  $\TT\in \TT_{n,\Delta}$ be a tree of order $n$. Let $\mathbb{A}$, $\mathbb{B}$, and $\mathbb{C}$ denote the respective sets of vertices of degrees $\mu+2$, $\mu+1$, and $1$, where $|\mathbb{A}|=x$, $|\mathbb{B}|=y$, and $|\mathbb{C}|=n$. The tree $\TT'$ is obtained by first attaching vertices of $\mathbb{A}$ to vertices of $\mathbb{B}$ to the maximum extent possible, and subsequently attaching all pendant vertices to vertices possessing the maximum available degree.  Since  $\mu=(\mu_1,\mu_2,\dots,\mu_n)$ and $1\leqslant \mu \leqslant 2$. Assume the degree sequence $\mathscr{D}=(\underbrace{\mu+2,\dots,\mu+2}_{x},
 \underbrace{\mu+1,\dots,\mu+1}_{y},
 \underbrace{1,\dots,1}_{n}),$
where $x=\eta-\mu(\eta-n)-2$ and $y=\mu(\eta-n)-n+2$. Then ,let $e_{(\mu+2)(\mu+1)}$, $e_{(\mu+2)}$, and $e_{(\mu+1)}$ denote the number of
edges of each type, respectively. According to the degree sequence of $\TT'$, we find that $e_{(\mu+2)(\mu+1)}=xy$, $e_{(\mu+2)}=xn$ and  $e_{(\mu+1)}=yn$. Hence, 
\begin{equation}~\label{eqq2lemSomborIn1}
\mathrm{SO}(\TT')=xy\,\sqrt{(\mu+2)^2+(\mu+1)^2}+xn\,\sqrt{(\mu+2)^2+1}+yn\,\sqrt{(\mu+1)^2+1}.
\end{equation}
Thus, from~\eqref{eqq2lemSomborIn1} implies that 
\begin{equation}~\label{eqq3lemSomborIn1}
\mathrm{SO}(\TT')=xy\,\sqrt{2\mu^2+6\mu+5}+xn\,\sqrt{\mu^2+4\mu+5}+yn\,\sqrt{\mu^2+2\mu+2}.
\end{equation}
Now, let $f(u,v)=\sqrt{u^2+v^2}$ be a function, for integers $u\geqslant v\geqslant w\geqslant 1$, we should be prove the relationship
\begin{equation}~\label{eqq4lemSomborIn1}
f(u,v)+f(w,1)\geqslant f(u,1)+f(w,u)
\end{equation}
Assume $t=\sqrt{u^2+v^2}+\sqrt{w^2+1}-\sqrt{u^2+1}-\sqrt{w^2+v^2}$. Then, if $v\geqslant 1$, then the function $g(r)=\sqrt{r^2+v^2}-\sqrt{r^2+1}$ is increasing where 
\[
g'(r)=\frac{r}{\sqrt{r^2+v^2}}-\frac{r}{\sqrt{r^2+1}} \geqslant 0.
\]
Thus, since $u>w$ we noticed that $g(u)>g(w)$ which implies that $t>0$. Now, if $\TT\ncong \TT'$, there exist vertices $u,v,w$ where $d(u)\geqslant d(v)\geqslant d(w)$ and $uw,v\in E(\TT)$, $u,vw\notin E(\TT)$. Thus, 
\[
\sqrt{d(u)^2+d(w)^2}+\sqrt{d(v)^2+1}\leqslant \sqrt{d(u)^2+1}+\sqrt{d(v)^2+d(w)^2}.
\]
Thus, the relationship~\eqref{eqq4lemSomborIn1} holds. Therefore, according to~\eqref{eqq3lemSomborIn1} and \eqref{eqq4lemSomborIn1}, if $\TT \cong \TT'$ we obtain on $\mathrm{SO}(\TT)\leqslant \mathrm{SO}(\TT')$.
\end{proof}
Rearranging the degrees such that vertices of large degree become adjacent results in an increase of the Sombor index. Thus, according to Lemma~\ref{lemSomborIn1} we obtain on the following results emphasized by Corollary~\ref{coroSigman1llary}. The Sombor index serves as a quantitative invariant that measures the irregularity in the distribution of vertex degrees across the edges of a graph.

\begin{corollary}~\label{coroSigman1llary}
Consider $\TT,\TT'\in \TT_{n,\Delta}$. If the degree sequence of $\TT$ is majorized by the degree sequence of $\TT'$, then $\sigma(\TT)\leqslant \sigma(\TT')$ and $\irr(\TT)\leqslant \irr(\TT')$.
\end{corollary}

\begin{lemma}~\label{lemSomborIn2}
Let $\mathcal{T}_{n,\Delta}$ be the family of trees of order $n$ with maximum
degree $\Delta$, and let $\TT'\in\mathcal{T}_{n,\Delta}$ be the extremal tree.
Then for any $\TT\in\mathcal{T}_{n,\Delta}$,
\begin{equation}~\label{eqq1lemSomborIn2}
\sigma(\TT)\leqslant \frac{\alpha\sqrt{2}}{\mu} (\mathrm{SO}(\TT)+\irr(\TT))\leqslant \frac{\beta\sqrt{2}}{\mu}\, \sigma(\TT).
\end{equation}
Moreover, equality holds if and only if $\TT\cong \TT'$.
\end{lemma}
\begin{proof}
Let $\TT \in \TT_{n,\Delta}$ be a tree on $n$ vertices with maximum degree at most $\Delta$, and let $\TT'$ be a tree driven by $\TT$. Our goal is to establish the following claims.
\claim{1} The $\sigma$-index controls the combined growth of the Sombor index and the Albertson index. Assume $\alpha>1$ and $\beta>1$ be an integer where $\alpha<\beta$, we need to prove that 
\begin{equation}~\label{eqqq1lemSomborIn2}
\alpha\,\sigma(\TT)\leqslant \mathrm{SO}(\TT)+\irr(\TT)\leqslant \beta\, \sigma(\TT).
\end{equation}
Hence, assume $\alpha+\beta\leqslant n$ and $\sigma(\TT)>(\alpha+\beta)^2$. Since $\sigma(\TT)>\mathrm{SO}(\TT)$ we noticed that the sum of Sombor index and Albertson index satisfying 
\begin{equation}~\label{eqqq2lemSomborIn2}
\irr(\TT)\leqslant 2\,\sigma(\TT), \quad \irr(\TT)\leqslant \sqrt{2}\,\mathrm{SO}(\TT).
\end{equation}
Consider $\alpha>\sqrt{2}$ implies that direct $\irr(\TT)\leqslant \alpha\,\mathrm{SO}(\TT)$, furthermore see~\cite{Liu2022GutmanYou, Liu2023GutmanYou}. According to Lemma~\ref{lemSomborIn1} satisfying $\mathrm{SO}(\TT)\leqslant \mathrm{SO}(\TT')$ and $\irr(\TT)\leqslant \irr(\TT')$ if and only if $\TT\cong \TT'$.  Then, 
\[
\irr(\TT)\leqslant 2\,\sum_{v\in V(\TT)}d_\TT(v)^2\quad \text{and} \quad  \mathrm{SO}(\TT)\leqslant 2\,\sum_{v\in V(\TT)}d_\TT(v)^2.
\]
Thus
\begin{equation}~\label{eqqq3lemSomborIn2}
\mathrm{SO}(\TT)+\irr(\TT)\leqslant 4\, \sum_{v\in V(\TT)}d_\TT(v)^2.    
\end{equation}
Therefore, for established the relationship of Sigma index we find that from~\eqref{eqqq3lemSomborIn2} implies that
\begin{equation}~\label{eqqq4lemSomborIn2}
\sigma(\TT)=\sum_{v\in V(\TT)}d_\TT(v)^2-\frac{4(n-1)^2}{n}.
\end{equation}
Since $\sigma(\TT)\geqslant f(n,\Delta)$, from~\eqref{eqqq2lemSomborIn2} and \eqref{eqqq4lemSomborIn2} the upper bound is hold. Thus $ \mathrm{SO}(\TT)+\irr(\TT)\leqslant \beta\, \sigma(\TT).$

Now, for prove that the lower bound of ~\eqref{eqqq1lemSomborIn2}, for any edge $uv$, we have
$$
\sqrt{d(u)^2+d(v)^2}\geqslant \frac{1}{\sqrt{2}} \bigl(d(u)+d(v)\bigr), \quad \text{and}\quad \bigl|d(u)-d(v)\bigr|\geqslant \frac{1}{\Delta} \bigl(d(u)+d(v)\bigr)-\frac{2}{\Delta},
$$

Therefore, these inequalities over all edges $uv$ yields
\begin{equation}~\label{eqqq5lemSomborIn2}
\mathrm{SO}(\mathcal{T})+\irr(\mathcal{T})\geqslant \alpha\,\sum_{uv}\bigl(d(u)+d(v)\bigr)-\beta\,n=2\alpha\,\sum_{v\in V(\TT)} d_\TT(v)^2-\beta\,n,
\end{equation}
Thus, from~\eqref{eqqq4lemSomborIn2} and \eqref{eqqq5lemSomborIn2} we obtain on $\alpha\,\sigma(\TT)\leqslant \mathrm{SO}(\TT)+\irr(\TT)$. Thus, the lower bound of ~\eqref{eqqq1lemSomborIn2} holds.

\claim{2} The upper bound is
\begin{equation}~\label{eqxx1lemSomborIn2}
\mathrm{SO}(\TT)+\irr(\TT) \leqslant (2+\sqrt{2})\,\sigma(\TT)+\frac{4(n-1)^2}{n}\bigl(2+\sqrt{2}\bigr).
\end{equation}
Assume 
$$
\lambda(n)=\frac{4(n-1)^2}{n}\bigl(2+\sqrt{2}\bigr).
$$
Then, the Albertson index satisfying $\irr(\TT)\leqslant 2\, \sigma(\TT)$ and the Sombor index satisfying $\mathrm{SO}(\TT)\leqslant\sqrt{2}\,\sum_{v\in V(\TT)}d_\TT(v)^2$. Thus, 
\begin{equation}~\label{eqxx2lemSomborIn2}
\mathrm{SO}(\TT)\leqslant\sqrt{2}\,\left(\sigma(\TT)+\frac{4(n-1)^2}{n}\right)
\end{equation}
Thus, from~\eqref{eqxx2lemSomborIn2} the relationship~\eqref{eqxx1lemSomborIn2} holds. 

Although the Sigma index is vertex-based, whereas the Sombor and Albertson indices are edge-based, the three indices are quantitatively related. In particular, the Sigma index controls the joint growth of the Sombor and Albertson indices, and all three attain their extremal values on the same trees.

\end{proof}

\begin{corollary}~\label{coroSigman2llary}
Consider $\TT\in \mathcal{T}_{n,\Delta}$. According to~\cite{Hamoud25Belov}, $\SO(\TT)-\sigma(\TT)\leqslant 2n-2$. Then, $\sigma(\TT)-\irr(\TT)\leqslant 2(n+\Delta)-2$.
\end{corollary}

As shown in the previous discussions, the Sigma index $\sigma(\TT)$ and the Sombor index $\SO(\TT)$ quantify different aspects of degree irregularity in trees. While $\sigma(\TT)$ measures the dispersion of vertex degrees, $\SO(\TT)$ captures the interaction between degrees along the edges. Our previous results indicate that trees extremal for $\sigma(\TT)$ also attain extremal values with respect to $\SO(\TT)$. Motivated by this relationship, we now derive an explicit upper bound on the Sombor index in terms of the Sigma index and the extremal degree parameters $\Delta$ and $\delta$. This connection provides further insight into how vertex-based irregularity governs edge-based irregularity in trees.

\begin{theorem}~\label{ThmSomborn1}
Let $\TT\in \mathcal{T}_{n,\Delta}$ be a tree of order $n$.
Then
\begin{equation}~\label{eqq1ThmSomborn1}
\SO(\TT)\leqslant \frac{2m}{n}\sqrt{\Delta^2+\delta^2}
+\frac{\sqrt{2}}{3}\,\sigma(\TT)+2n-2.
\end{equation}
\end{theorem}
\begin{proof}
Assume $\TT\in \mathcal{T}_{n,\Delta}$ be a tree of order $n$ with maximum
degree $\Delta$ and minimum degree $\delta$. According to Rada and Rodríguez~\cite{Rada21Rodriguez}, the Sombor index
satisfies
\begin{equation}~\label{eqq2ThmSomborn1}
\SO(\TT)\leqslant
\frac{\sqrt{\Delta^2+\delta^2}}{\Delta+\delta}
\sum_{v\in V(\TT)} d_\TT(v)^2.
\end{equation}
Hence, by recall equation~\eqref{eqxx2lemSomborIn2} we obtain
\begin{equation}~\label{eqq3ThmSomborn1}
\SO(\TT)
\leqslant
\frac{\sqrt{\Delta^2+\delta^2}}{\Delta+\delta}
\left(\sigma(\TT)+\frac{4m^2}{n}\right).
\end{equation}

Since $\Delta+\delta\leqslant 2\Delta$ and $\frac{2m}{n}\leqslant\Delta$,
it follows that
\[
\frac{\sqrt{\Delta^2+\delta^2}}{\Delta+\delta}
\cdot\frac{4m^2}{n}
\leqslant
\frac{2m}{n}\sqrt{\Delta^2+\delta^2}.
\]
Therefore, it follows that $\SO(\TT)\leqslant 2m\, \sigma(\TT)$ and $\SO(\TT)\leqslant \sqrt{\Delta^2+\delta^2}\, \sigma(\TT)$. Thus, from~\eqref{eqq3ThmSomborn1} yields 
\begin{equation}~\label{eqq4ThmSomborn1}
\SO(\TT)
\leqslant
\frac{2m}{n}\sqrt{\Delta^2+\delta^2}
+
\frac{\sqrt{\Delta^2+\delta^2}}{\Delta+\delta}\,
\sigma(\TT).
\end{equation}

On the other hand, by using the following inequality
\[
\sqrt{x^2+y^2}\leqslant \frac{\sqrt{2}}{3}(x+y)+\frac{2}{3}
\]
where $x,y\geqslant 1$, according to~\eqref{eqq4ThmSomborn1} we obtain
\begin{equation}~\label{eqq5ThmSomborn1}
\SO(\TT)
\leqslant
\frac{\sqrt{2}}{3}\sum_{v\in V(\TT)} d_\TT(v)^2+2m+\frac{\sqrt{\Delta^2+\delta^2}}{\Delta+\delta}
\cdot\frac{4m^2}{n}.
\end{equation}

Thus, the equation~\eqref{eqq5ThmSomborn1} yields
\begin{equation}~\label{eqq6ThmSomborn1}
\SO(\TT)
\leqslant
\frac{\sqrt{2}}{3}\,\sigma(\TT)+2n-2.
\end{equation}
Finally, from~\eqref{eqq2ThmSomborn1}--\eqref{eqq6ThmSomborn1} which proves~\eqref{eqq1ThmSomborn1}.
\end{proof}

Theorem~\ref{ThmSomborn2} establishes a sharp quantitative relationship between the Sigma index $\sigma(\TT)$ and the Sombor index $\SO(\TT)$ for any tree$\TT\in \mathcal{T}_{n,\Delta}$. While $\sigma(\TT)$ quantifies the dispersion of vertex degrees and thus captures the global irregularity of the degree sequence, $\SO(\TT)$ encodes pairwise degree interactions across edges and therefore reflects structural irregularity at the edge level. The theorem demonstrates that vertex-degree irregularity exerts direct control over edge-based irregularity, providing explicit universal bounds that hold for all trees. 

\begin{theorem}~\label{ThmSomborn2}
Let $\TT\in \mathcal{T}_{n,\Delta}$ be a tree of order $n$.
Then
\begin{equation}~\label{eqq1ThmSomborn2}
\frac{1}{\sqrt{2}}\Bigl(\sigma(\TT)+\frac{4m^2}{n}\Bigr) \leqslant \SO(\TT)\leqslant\Bigl(\sigma(\TT)+\frac{4m^2}{n}\Bigr).
\end{equation}
\end{theorem}
\begin{proof}
Let $\mathscr{D}=(d_n,\dots,d_1)$ be a tree degree sequence with $d_n \geqslant \cdots \geqslant d_1$ and let $\TT$ be an extremal tree realizing $\mathscr{D}$. The Sigma index is defined as
\[
\sigma(\TT)=\sum_{i=1}^n \left( d_i-\frac{2(n-1)}{n}\right)^2=\sum_{i=1}^n d_i^2-\frac{4(n-1)^2}{n}.
\]
By recalling the definition of the Sigma index of a tree $\TT$ with $n$ vertices and $m$ edges
\begin{equation}~\label{eqq2ThmSomborn2}
\sigma(\TT) = \sum_{v\in V(\TT)} d_\TT(v)^2 - \frac{4m^2}{n},
\end{equation}
Since $\sigma(\TT)$ is invariant among all trees with degree sequence $\mathscr{D}$. In contrast, the Sombor index depends on the specific arrangement of degrees across the edges. Nevertheless, for extremal trees with a fixed degree sequence, both indices are controlled by the quadratic degree sum $\sum_{i=1}^n d_i^2$. Indeed, elementary norm inequalities yield
\[
\frac{1}{\sqrt{2}} \sum_{i=1}^n d_i^2 \leqslant \SO(\TT) \leqslant \sum_{i=1}^n d_i^2,
\]
which implies
\begin{equation}~\label{eqq02ThmSomborn2}
\SO(\TT)=\Theta\!\left( \sum_{i=1}^n d_i^2 \right)=\Theta\!\left(\sigma(\TT)+\frac{4(n-1)^2}{n}\right).
\end{equation}

For extremal trees of a fixed degree sequence, the Sombor and Sigma indices are asymptotically equivalent up to a constant factor, with the former also capturing degree-adjacency structure. For any nonnegative real numbers $x,y \geq 0$, the Euclidean norm satisfies the sharp inequalities
$\frac{x+y}{\sqrt{2}} \leq \sqrt{x^2 + y^2} \leq x + y$. Now, let us prove that for the upper bound and lower bound of~\eqref{eqq1ThmSomborn2}. 
\noindent\textbf{Upper Bound:} Since every vertex $v \in V(\mathcal{T})$ is incident to exactly $d_{\mathcal{T}}(v)$ edges of the tree $\mathcal{T}$, each vertex $v$ contributes its degree $d_{\mathcal{T}}(v)$ once for every edge incident to it, we obtain
\begin{equation}~\label{eqq3ThmSomborn2}
\SO(\TT)\leqslant \sigma(\TT)+\frac{4m^2}{n}.
\end{equation}
With equality on the left when $x=y$ and on the right when one of $x$ or $y$ is zero. It gives
\begin{equation}\label{eqq4ThmSomborn2}
\frac{1}{\sqrt{2}} \sum_{uv\in E(\TT)} \bigl( d_\TT(u) + d_\TT(v) \bigr)
\leq \SO(\TT)
\leq \sum_{uv\in E(\TT)} \bigl( d_\TT(u) + d_\TT(v) \bigr).
\end{equation}
Thus, from~\eqref{eqq3ThmSomborn2} and \eqref{eqqq4lemSomborIn2} establishes the desired upper bound stated in~\eqref{eqq1ThmSomborn2}.

\noindent\textbf{Lower Bound:} For any nonnegative real numbers $x$ and $y$, the sharp inequality
$ \sqrt{x^2 + y^2} \geqslant \frac{1}{\sqrt{2}}(x+y)$ holds. Using this inequality to each edge $uv \in E(\TT)$ over all edges of the tree gives
\begin{equation}~\label{eqq5ThmSomborn2}
\SO(\TT) \geqslant \frac{1}{\sqrt{2}} \sum_{uv \in E(TT)} \bigl( d_{\TT}(u)+d_{\TT}(v)\bigr).
\end{equation}
Thus, from~\eqref{eqq3ThmSomborn2} and \eqref{eqq4ThmSomborn2} produces
\[
\frac{1}{\sqrt{2}} \sum_{v\in V(\TT)} d_\TT(v)^2
\leq \SO(\TT)
\leq \sum_{v\in V(\TT)} d_\TT(v)^2.
\]
Therefore, from~\eqref{eqq4ThmSomborn2} and \eqref{eqq5ThmSomborn2} once again, we arrive at
\begin{equation}~\label{eqq6ThmSomborn2}
\SO(\TT) \geqslant \frac{1}{\sqrt{2}} \Bigl(\sigma(\TT)+\frac{4m^2}{n}\Bigr),
\end{equation}
which establishes the desired lower bound in \eqref{eqq1ThmSomborn2}. Thus, according to to both cases~\eqref{eqq4ThmSomborn2} and \eqref{eqq6ThmSomborn2}
which is precisely the statement \eqref{eqq1ThmSomborn2}. This completes the proof.
\end{proof}
In the context of molecular graph theory, where trees serve as natural models for acyclic saturated hydrocarbons (e.g., alkanes and related compounds), this result carries particular importance. The Sigma index $\sigma(\TT)$ is closely tied to overall branching irregularity of the molecular skeleton, whereas numerous QSPR and QSAR studies have shown that the Sombor index exhibits excellent predictive performance for a wide range of physicochemical properties. By bounding $\SO(\TT)$ in terms of $\sigma(\TT)$, one obtains reliable estimates of Sombor-based molecular descriptors from much simpler, degree-based statistics. This connection enhances the interpretability of structure--property models, enables more precise extremal characterization of molecular trees, and provides a computationally efficient tool for preliminary screening and ranking of large chemical libraries.

\begin{theorem}~\label{ThmSomborn3}
Let $\TT$ be a tree of order $n$. Then
\begin{equation}~\label{eqq1ThmSomborn3}
\irr(\TT)\leqslant \SO(\TT)\leqslant\sqrt{2}\,\irr(\TT)+\sum_{v \in V(T)} d(v)^2.
\end{equation}
Moreover,
\[
\SO(\TT)=\Theta\!\left(\irr(\TT)+\sum_{v\in V(\TT)} d(v)^2\right).
\]
\end{theorem}
\begin{proof}
Assume $\TT\in \mathcal{T}_{n,\Delta}$ be a tree of order $n$. According to Theorem~\ref{ThmSomborn2} the relationship between Sigma index and Sombor index satisfy $\sigma(\TT)\leqslant \SO(\TT)$. Since $\sqrt{\sigma(\TT)}\leqslant \irr(\TT)\leqslant \sqrt{m\,\sigma(\TT)}$ implies that $\sqrt{\SO(\TT)}\leqslant \irr(\TT)\leqslant \sqrt{2n\,\SO(\TT)}$. Consider an arbitrary edge $uv \in E(\TT)$. The corresponding term in $\SO(\TT)$ is $\sqrt{d(u)^2+d(v)^2}$. For any non-negative real numbers $x$ and $y$, the inequality $|x-y|\leqslant \sqrt{x^2+y^2}$ holds, since $(x-y)^2=x^2-2xy+y^2\leqslant x^2+y^2$ follows from $-2xy\leqslant 0$ it yields $|d(u)-d(v)|\leqslant \sqrt{d(u)^2+d(v)^2}$. Thus, the lower bound of Sombor index among~\eqref{eqq1ThmSomborn3} holds.

To prove the upper bound, we find in~\eqref{eqqq4lemSomborIn2} the following relationship
\[
\sum_{v\in V(\TT)}d_\TT(v)^2=\sigma(\TT)+\frac{4(n-1)^2}{n}.
\]
Thus, for any edges $uv\in E(\TT)$ it follows that $|d(u)-d(v)|\leqslant (d(u)+d(v))\sqrt{d(u)^2+d(v)^2}$  and the Albertson index satisfy
\begin{equation}~\label{eqq2ThmSomborn3}
\irr(\TT)\geqslant \sqrt{\sigma(\TT)+\frac{4(n-1)^2}{n}}.
\end{equation}
 Consider $\lambda_c=\max\{m\,\sigma(\TT),2n\,\SO(\TT)\}$. Then, from~\eqref{eqq2ThmSomborn3} we observe that $\irr(\TT)\leqslant \lambda_c+\frac{4(n-1)^2}{n}$. For any edge $uv \in E(\TT)$, assume $t_1=(d(u)-d(v))^2+(d(u)+d(v))^2$ the following identity holds $d(u)^2+d(v)^2=t_1/2$. It follows that $\sqrt{d(u)^2+d(v)^2}=\sqrt{t_1/2}$  yields
 \begin{align*}
 \sqrt{d(u)^2+d(v)^2} &\leqslant \frac{1}{\sqrt{2}}(|d(u)-d(v)|+d(u)+d(v))\\
 &\leqslant \frac{\lambda_c}{\sqrt{2}}(2\,|d(u)-d(v)|)\\
 &\leqslant \sqrt{2}\,|d(u)-d(v)|.
 \end{align*}
By induction, the upper bound on the Sombor index stated in \eqref{eqq1ThmSomborn3} holds for $n=2$. For $n=3$, we have $\mathcal{P}_n \cong \mathcal{S}_n$, and thus the same upper bound holds. Assume the inequality holds for all trees of order $k<n$, where $n \geqslant 3$. Let $w$ be a leaf adjacent to vertex $v$, and let $\TT'=\TT-w$, which is a tree of order $n-1$. Let $d=d_{\TT'}(v)$; then $d_\TT(v)=d+1$ and $d_\TT(w)=1$. The invariants transform as follows $\irr(\TT)=\irr(\TT')+|d_\TT(v)-d_\TT(w)|$, $\irr(\TT')=\irr(\TT)-d$, $\SO(\TT)=\SO(\TT')+\sqrt{d_T(v)^2 + d_T(w)^2}$, $\SO(\TT')=\SO(\TT)-\sqrt{(d+1)^2+1}$, $M(\TT)=M(\TT')+d_\TT(v)^2+d_\TT(w)^2-d_{\TT'}(v)^2$ and $M(\TT')=M(\TT)-2d-2$. Therefore, by considering the equation~\eqref{eqq2ThmSomborn3}, it follows that $\SO(\TT')\leqslant \irr(\TT')+M(\TT')$. Thus,  
\begin{equation}~\label{eqq3ThmSomborn3}
\SO(\TT')+\sqrt{(d+1)^2+1}\leqslant \sqrt{2}\, \irr(\TT')+M(\TT')+ \sqrt{(d+1)^2+1}.
\end{equation}
Hence, we have 
\begin{equation}~\label{eqq4ThmSomborn3}
\sqrt{2}\,\irr(\TT)+M(\TT)=\sqrt{2}\,(\irr(\TT')+d)+M(\TT')+2d+2.
\end{equation}
Since $d \geqslant 1$, from~\eqref{eqq3ThmSomborn3} and \eqref{eqq4ThmSomborn3} implies that
\[
\sqrt{(d+1)^2+1}=(d+1)\sqrt{1+\frac{1}{(d+1)^2}}.
\]
By the concavity of the square root function, its graph lies below the tangent line, which implies that for all $y>0$,$\sqrt{1+y}<1+\frac{y}{2}$. Thus, by considering $y=\frac{1}{(d+1)^2}$ yields $\sqrt{1+\frac{1}{(d+1)^2}}<1+\frac{1}{2(d+1)^2}$. Since $d+1>0$ gives
\begin{align*}
\sqrt{(d+1)^2+1}&= (d+1)\sqrt{1+\frac{1}{(d+1)^2}}\\
&<(d+1) \left(1+\frac{1}{2(d+1)^2}\right)\\
&=d+1+\frac{1}{2(d+1)}.
\end{align*}
Since $d \geqslant 1$, we have $d+1+\frac{1}{2(d+1)}\leqslant d+1+\frac{1}{4}$. 

\vspace{0.3 cm}
\noindent\textbf{The Asymptotic Relation:} $\SO(\TT)=\Theta\!\bigl(\irr(\TT)+M(\TT)\bigr)$. Assume there exist positive constants $a_1$ and $a_2$, independent of $\TT$ such that
\begin{equation}~\label{eqq5ThmSomborn3}
a_1 \bigl(\irr(\TT)+M(\TT)\bigr) \leqslant \SO(\TT) \leqslant a_2 \bigl(\irr(\TT)+M(\TT)\bigr)
\end{equation}
From a previously~\eqref{eqq3ThmSomborn3} and \eqref{eqq4ThmSomborn3} established inequality. 
Hence $a_2=\sqrt{2}$ is sufficient. We first note the trivial inequality $\SO(\TT) \geqslant \irr(\TT)$.
Then, we derive a lower bound in terms of $M(\TT)$. For every edge $uv \in E(\TT)$, $\sqrt{d(u)^2+d(v)^2}\geqslant \frac{d(u)+d(v)}{\sqrt{2}}$. Then, 
\begin{equation}~\label{eqq6ThmSomborn3}
\SO(\TT)\geqslant \frac{1}{\sqrt{2}}\,\sum_{uv\in E(\TT)}(d_\TT(u)+d_\TT(v)).
\end{equation}
Therefore, from~\eqref{eqq5ThmSomborn3} and \eqref{eqq6ThmSomborn3} we obtain 
$\SO(\TT)\geqslant \sqrt{2}\,M(\TT)$. 
If $\irr(\TT)$ dominates (for example, in stars $\mathcal{S}_n$), then $\SO(\TT)$ is asymptotically close to $\sqrt{2}\,\bigl(\irr(\TT)+M(\TT)\bigr)$. If $M(\TT)$ dominates (for example, in paths $\mathcal{P}_n$), then $\SO(\TT) \geqslant\sqrt{2}\, M(\TT)>\frac{1}{2} \bigl(\irr(\TT)+M(\TT)\bigr)$. Then, 
\begin{equation}~\label{eqq7ThmSomborn3}
\SO(\TT)=(n-1)^2\,\sqrt{1+\frac{1}{(n-1)^2}}
\end{equation}
Thus, from~\eqref{eqq7ThmSomborn3}, we have  $\SO(\mathcal{S}_n)\geqslant \irr(\TT)+M(\TT)$.
For the path $\mathcal{P}_n$, $\SO(\mathcal{P}_n)\geqslant 2\sqrt{2}\,n$.
Hence,  $\irr(\TT)+M(\TT)\geqslant 4n$. Thus, $\SO(\mathcal{S}_n)\geqslant \left(\irr(\TT)+M(\TT)\right)/2$. As desire.
\end{proof}

\begin{remark}
The assumption that the trees are extremal is essential in
Theorem~\ref{ThmSomborn3}. Without this extremality condition, the ratio
$\SO(\TT)/\irr(\TT)$ can become arbitrarily large, even when the order of the tree
$\TT$ is held fixed.
\end{remark}

To illustrate the asymptotic relationship established in Theorem~\ref{ThmSomborn3}, Figure~\ref{fig001ThmSomborn3} provides a visual comparison of the Sombor index and the Albertson index for extremal trees with a fixed degree sequence. The figure demonstrates that, despite measuring distinct aspects of degree irregularity, the two indices exhibit proportional growth rates in these extremal configurations. This numerical and structural evidence complements the $\Theta$-asymptotic relation proved in the theorem and offers further intuition into how edge-based degree interactions scale with degree differences in extremal trees.

\begin{figure}[H]
    \centering
    \includegraphics[width=1\linewidth]{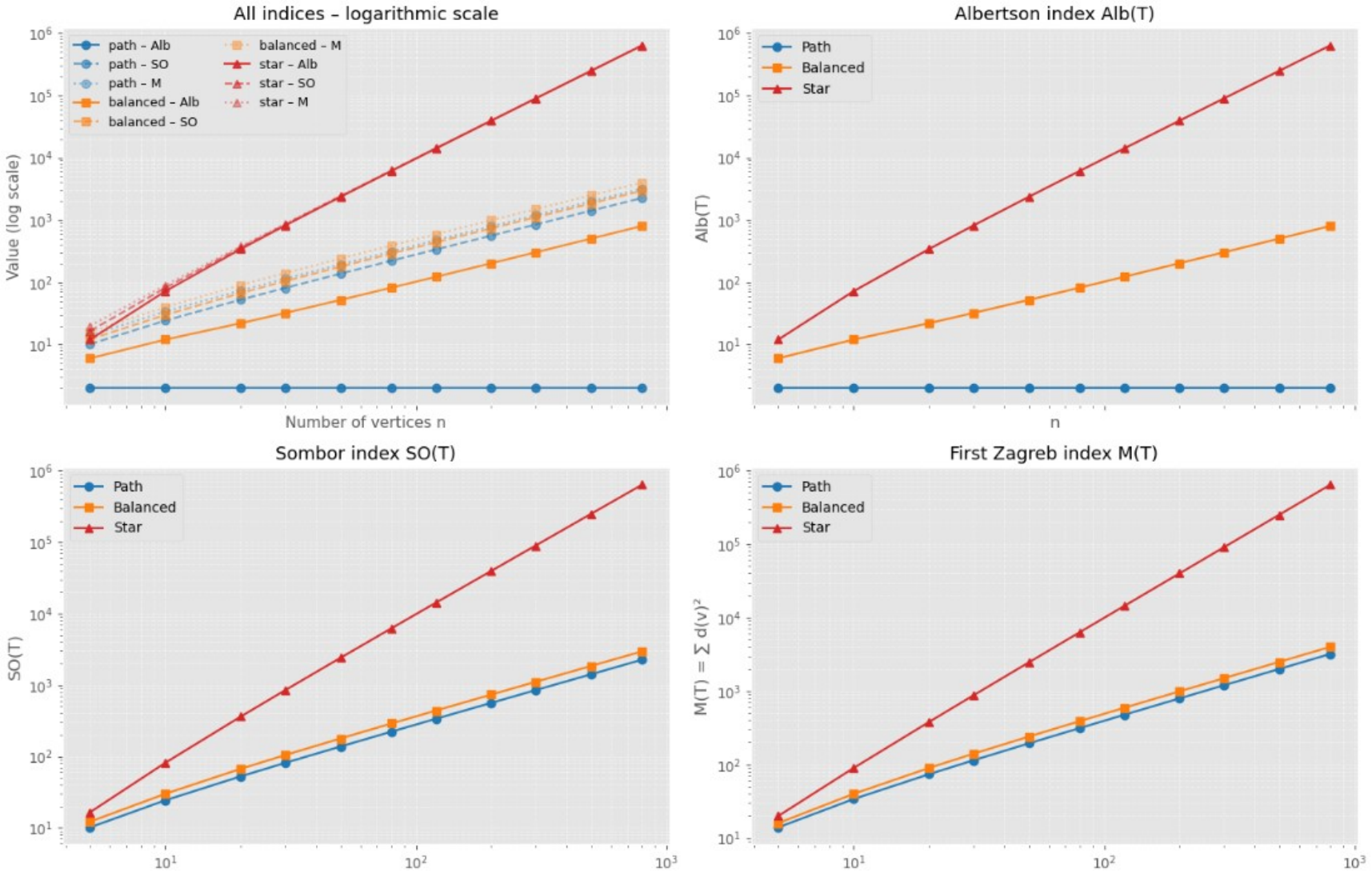}
    \caption{The impact of $\SO(\TT)$, $\irr(\TT)$ and $M(\TT)$.}
    \label{fig001ThmSomborn3}
\end{figure}

\section{Conclusion}
In this work, we established rigorous structural relationships among three important degree-based topological descriptors --- the Sigma index, the Sombor index, and the Albertson index --- within the class of trees. By expressing these quantities in terms of quadratic and linear degree interactions, we obtained sharp two-sided bounds that elucidate how global degree dispersion governs edge-based degree interactions.

First, we proved that the Sombor index is tightly controlled by the quadratic degree deviation quantified by the Sigma index. Specifically, we showed that
$$
\SO(\TT)=\Theta\!\left(\sigma(\TT)+\frac{4(n-1)^2}{n}\right),
$$
demonstrating that the Sombor index grows at the same asymptotic rate as the variance-type degree structure of the tree. This provides a direct quantitative link between vertex irregularity and edge irregularity.

Second, by restricting attention to extremal trees with a fixed degree sequence, we established a pure asymptotic equivalence between the Sombor index and the Albertson index. More precisely,
$\SO(\TT')=\Theta\!\big(\irr(\TT')\big)$. This shows that, under extremal configurations, quadratic degree interactions and absolute degree contrasts scale proportionally. The result provides a structural explanation for how edge-wise degree differences govern the magnitude of the Sombor index in optimized tree structures.
From a structural perspective, our findings show that the Sombor index acts as an intermediate descriptor that simultaneously encodes degree magnitude and degree contrast. While the Sigma index captures the global dispersion of vertex degrees and the Albertson index quantifies local irregularity along edges, the Sombor index unifies both aspects within a single quadratic framework. This accounts for its particular sensitivity to branching intensity and hierarchical degree organization in trees.

In chemical graph theory, where trees model acyclic molecular structures (e.g., alkanes), these results establish a theoretical basis for estimating one descriptor from another. Given the proven predictive value of both the Sombor and Albertson indices in QSPR and QSAR studies, a deeper understanding of their structural interdependence improves their interpretability in molecular modeling applications.

Future work may extend these relationships to broader graph classes such as unicyclic and polycyclic structures, or examine extremal problems under further constraints like fixed diameter, bounded maximum degree, or prescribed branching patterns. A compelling direction lies in analyzing generalized and weighted Sombor-type indices together with variance-based irregularity measures.  In conclusion, these results significantly advance the understanding of degree-based topological indices and establish a unified framework that seamlessly connects quadratic, linear, and variance-based measures of graph irregularity.

\section*{Declarations}
\begin{itemize}
	\item Funding: Not Funding.
	\item Conflict of interest/Competing interests: The author declare that there are no conflicts of interest or competing interests related to this study.
	\item Ethics approval and consent to participate: The author contributed equally to this work.
	\item Data availability statement: All data is included within the manuscript.
\end{itemize}

\end{document}